\documentclass[11pt,a4paper]{amsart}
\usepackage[T1]{fontenc}
\usepackage{lmodern}
\usepackage{amsmath,amssymb,amsthm,mathtools}
\usepackage{a4wide}
\usepackage{microtype}
\usepackage{needspace}
\usepackage{cite}
\usepackage{url}
\usepackage{advdate}
\usepackage[colorlinks=true,linkcolor=blue,citecolor=blue,urlcolor=blue]{hyperref}
\usepackage[nameinlink,noabbrev]{cleveref}
\numberwithin{equation}{section}
\newtheorem{theorem}{Theorem}[section]
\newtheorem{proposition}[theorem]{Proposition}
\newtheorem{lemma}[theorem]{Lemma}
\newtheorem{corollary}[theorem]{Corollary}
\theoremstyle{definition}
\newtheorem{definition}[theorem]{Definition}
\newtheorem{question}[theorem]{Question}

\theoremstyle{remark}
\newtheorem{remark}[theorem]{Remark}
\theoremstyle{plain}
\newtheorem{maintheorem}{Theorem}

\DeclareMathOperator{\Hom}{Hom}
\DeclareMathOperator{\FBL}{FBL}
\DeclareMathOperator{\id}{id}
\DeclareMathOperator{\supp}{supp}
\newcommand{\N}{\mathbb N}
\newcommand{\R}{\mathbb R}
\newcommand{\eps}{\varepsilon}

\begin{document}
\title[Semiprojective Banach lattices]{Semiprojective Banach lattices}
\author[T.~Kania]{Tomasz Kania}
\address[T.~Kania]{Mathematical Institute\\Czech Academy of Sciences\\\v Zitn\'a 25 \\115 67 Praha 1\\Czech Republic and Institute of Mathematics and Computer Science\\Jagiellonian University\\{\L}ojasiewicza 6, 30-348 Krak\'{o}w, Poland}
\email{kania@math.cas.cz, tomasz.marcin.kania@gmail.com}
\thanks{RVO: 67985840.}
\author[M.~Niwi\'nski]{Mariusz Niwi\'nski}
\address[M.~Niwi\'nski]{Doctoral School of Exact and Natural Sciences\\Faculty of Mathematics and Computer Science\\Institute of Mathematics, Jagiellonian University\\{\L}ojasiewicza 6, 30-348 Krak\'ow, Poland}
\email{mariusz.niwinski@doctoral.uj.edu.pl}
\date{\AdvanceDate[-14]\today}
\subjclass[2020]{46B42, 46E15, 54C15, 18A30}
\keywords{Banach lattice, semiprojectivity, projectivity, absolute neighbourhood retract, lattice homomorphism, free Banach lattice}
\begin{abstract}
We study a norm-controlled notion of semiprojectivity for Banach lattices, testing sequential quotient towers in separable ambient lattices and requiring arbitrarily small norm loss at a fixed stage. For a locally compact second-countable Hausdorff space $X$, we prove that $C_0(X)$ is semiprojective if and only if $X$ is compact and an absolute neighbourhood retract. Finite $\ell_1$-sums preserve semiprojectivity, and arbitrary $\ell_1$-sums of $1^+$-projective lattices are semiprojective without any order-unit assumption. By contrast, $\R\oplus_\infty\ell_1$ fails to be semiprojective, showing that neither semiprojectivity nor $1^+$-projectivity is preserved by finite $\ell_\infty$-sums, extensions, or pullbacks over a projective base. Every separable semiprojective lattice is $\lambda$-projective for some finite $\lambda$, and its scalar-valued lattice homomorphisms separate points. Consequently, a separable reflexive Banach lattice is semiprojective precisely when it is finite-dimensional, and no nonzero separable atomless order-continuous lattice is semiprojective. We also characterise $\ell$-openness as $1^+$-projectivity relative to separable ambient lattices; for separable domains this is ordinary $1^+$-projectivity. The restriction to separable ambient lattices is redundant for separable domains but affects the nonseparable examples.
\end{abstract}
\maketitle

\section{Introduction}
Projectivity expresses the possibility of lifting a morphism through a quotient. In the category of Banach lattices and bounded lattice homomorphisms, a useful metric version permits an arbitrarily small multiplicative increase in the norm of a lift. De Pagter and Wickstead introduced this property, now called $1^+$-projectivity, and established its connection with free Banach lattices \cite{dpw}. Their examples include finite-dimensional Banach lattices and $\ell_1$. Avil\'es, Mart\'{\i}nez-Cervantes and Rodr\'{\i}guez Abell\'an subsequently proved that $C(K)$ is $1^+$-projective exactly when $K$ is an absolute neighbourhood retract \cite{avilaes2020}.

Semiprojectivity replaces a single arbitrary lifting problem by eventual lifting along a prescribed sequential tower of quotients. The notion originates in the theory of $C^*$-algebras \cite{blackadar85,loring}, and its dependence on the ambient category is discussed in \cite{thiel-categories}. We consider the following Banach-lattice version. Given a separable Banach lattice $B$, increasing closed ideals $I_n$ with $I=\overline{\bigcup_n I_n}$, and a contractive lattice homomorphism $T:A\to B/I$, we require a stage $n$ at which $T$ has lifts into $B/I_n$ of norm at most $1+\eps$ for every $\eps>0$. The stage must be independent of $\eps$. The definition places no separability assumption on $A$.

Two features distinguish this setting. First, multiplication by a positive scalar preserves lattice homomorphisms. A path-space construction therefore converts eventual lifting into exact bounded lifting through quotients of separable lattices. Secondly, disjoint elements in a quotient can be lifted disjointly when the family is countable. Applying this fact to order units of separable \emph{image ideals}, rather than to order units in the domain, yields direct-sum results without restrictions on the summands' order units.

Our topological characterisation takes the following form.
\begin{maintheorem}\label{main:A}
Let $X$ be a locally compact second-countable Hausdorff space. Then $C_0(X)$ is semiprojective if and only if $X$ is compact and an absolute neighbourhood retract.
\end{maintheorem}
For compact $X$, sufficiency follows from the projectivity theorem in \cite{avilaes2020}. For necessity, a lattice section of a restriction map between $C(K)$-spaces produces a neighbourhood retraction, through the weighted-composition representation of lattice homomorphisms. If $X$ is noncompact, then $c_0$ is a contractive lattice retract of $C_0(X)$, and the obstruction for $c_0$ applies. In particular, semiprojectivity need not pass to closed ideals or to quotient lattices.

The direct-sum results are as follows. Here and below, $1^+$-projectivity \emph{relative to separable ambient lattices} means that the usual lifting property is tested only on quotients of separable Banach lattices.
\begin{maintheorem}\label{main:B}
The following assertions hold.
\begin{enumerate}
\item Finite $\ell_1$-sums of semiprojective Banach lattices are semiprojective.
\item An arbitrary $\ell_1$-sum of Banach lattices that are $1^+$-projective relative to separable ambient lattices has the same relative projectivity property. In particular, an arbitrary $\ell_1$-sum of $1^+$-projective Banach lattices is semiprojective.
\item The Banach lattice $\R\oplus_\infty\ell_1$ is not semiprojective. Consequently, neither semiprojectivity nor $1^+$-projectivity is preserved by finite $\ell_\infty$-sums, by extensions of Banach lattices, or by pullbacks along quotient maps between $1^+$-projective lattices.
\end{enumerate}
\end{maintheorem}
No topological order units are assumed in this theorem. The stronger assertion that a countable $\ell_1$-sum is $1^+$-projective in the unrestricted sense does hold when its summands are $1^+$-projective and have topological order units, by \cite[Theorem~11.8]{dpw}. The distinction matters: $\ell_1(\Gamma)$ is semiprojective for every set $\Gamma$, whereas it is $1^+$-projective exactly when $\Gamma$ is countable. A further example is $\R\oplus_1\FBL(\Gamma)$ for uncountable $\Gamma$. By contrast, assertion~(3) shows that replacing the $\ell_1$-norm by the $\ell_\infty$-norm destroys semiprojectivity even for $\R\oplus_\infty\ell_1$, which is $2^+$-projective.

These examples distinguish unrestricted projectivity from our semiprojectivity property, which uses separable ambient lattices. In fact, they already have full $1^+$-lifting relative to those ambient lattices. They therefore do not separate the two lifting properties when both are tested against precisely the same class of ambient objects. For separable domains, we prove that the ambient-separability restriction can be removed, both for projectivity and for semiprojectivity.

The path-space argument also leads to structural restrictions. A Banach lattice is called $\infty$-projective if it is $\lambda$-projective for some finite $\lambda>1$, in the terminology of \cite{avilaes2020}.
\begin{maintheorem}\label{main:C}
Let $A$ be a separable semiprojective Banach lattice. Then $A$ is $\infty$-projective and scalar-valued lattice homomorphisms separate its points. If its norm is order continuous, then $A$ is atomic; if in addition $A$ is infinite-dimensional, it contains a lattice copy of $\ell_1$. Consequently, a separable reflexive Banach lattice is semiprojective if and only if it is finite-dimensional.
\end{maintheorem}
The implication to $\infty$-projectivity allows us to apply the obstruction theorems of \cite{avilaes2020}. Besides recovering the failure of semiprojectivity for $c_0$ and $\ell_p$, $1<p<\infty$, we obtain the failure for every nonzero separable atomless order-continuous lattice. This includes the classical atomless $L_p$-spaces and separable Orlicz hearts. We also obtain an obstruction for free Banach lattices generated by separable Banach spaces without the Schur property. None of these arguments removes the loss in the lifting constant, so the question whether a separable semiprojective lattice must be $1^+$-projective remains open here.

Finally, motivated by \cite{oyetunbi-tikuisis}, call $A$ \emph{$\ell$-open} if, for every quotient of a separable lattice, the contractive lattice homomorphisms that admit lifts of norm at most $1+\eps$ for every $\eps>0$ form an open set in the point-norm topology.
\begin{maintheorem}\label{main:D}
A Banach lattice is $\ell$-open if and only if it is $1^+$-projective relative to separable ambient lattices. In particular, every $\ell$-open lattice is semiprojective and $\ell$-closed. For separable lattices, $\ell$-openness is equivalent to ordinary $1^+$-projectivity.
\end{maintheorem}
The proof uses the radial path $t\mapsto tT$. If the infimum of the norms of lifts of $T$ were greater than one, this path would meet a liftable map at the boundary of the set of liftable maps, contradicting openness.

For comparison, for compact metric $X$, S\o rensen and Thiel proved that the $C^*$-algebra $C(X)$ is semiprojective if and only if $X$ is an ANR of dimension at most one \cite{sorensen-thiel}. In the Banach-lattice characterisation there is no dimensional restriction. The relevant morphisms and lifting obstructions are different: weighted composition captures the topological obstruction between $C(K)$-lattices, whereas lifting a homomorphism into a general $C^*$-algebra must also respect commutation relations. The projectivity comparison requires care about units: in the category of unital $C^*$-algebras and unital $*$-homomorphisms, $C(X)$ is projective exactly when $X$ is an AR of dimension at most one \cite{chigogidze-dranishnikov}.

\section{Definitions and elementary permanence}\label{sec:pre}
All Banach lattices and linear spaces in this paper are real. A lattice homomorphism is a bounded linear operator $T$ satisfying $T|x|=|Tx|$. A closed ideal always means a norm-closed lattice ideal. Quotients carry their usual Banach-lattice quotient norms. For a compact Hausdorff space $K$, $C(K)$ has the supremum norm and pointwise order; for a locally compact Hausdorff space $X$, $C_0(X)$ consists of the continuous functions vanishing at infinity. We write $I_F=\{f\in C(K):f|_F=0\}$ for closed $F\subseteq K$.

A metrisable space is an \emph{absolute neighbourhood retract} (ANR) if, whenever it is embedded as a closed subset of a metrisable space, it is a retract of some open neighbourhood.

We use the standard terminology for Banach lattices from \cite{meyer-nieberg}. A norm is \emph{order continuous} if $x_\alpha\downarrow0$ implies $\|x_\alpha\|\to0$. We write $H^d=\{x:|x|\wedge|h|=0\text{ for all }h\in H\}$ for the disjoint complement of a set $H$. The free Banach lattice $\FBL(S)$ generated by a set $S$ has the universal property that every bounded map from its generators into a Banach lattice extends uniquely to a lattice homomorphism with the same norm \cite{dpw}. For a Banach space $E$, the notation $\FBL[E]$ instead refers to the universal extension of bounded \emph{linear} operators from the canonical isometric copy of $E$ \cite{aviles-fblE}.

\begin{definition}\label{pre:semi}
A Banach lattice $A$ is \emph{semiprojective} if the following holds. Let $B$ be a separable Banach lattice, let $(I_n)_{n\in\N}$ be an increasing sequence of closed ideals of $B$, and put $I=\overline{\bigcup_n I_n}$. Write $\pi_n:B/I_n\to B/I$ for the canonical quotient map. For every contractive lattice homomorphism $T:A\to B/I$, there exists $n\in\N$ such that for every $\eps>0$ there is a lattice homomorphism $S_\eps:A\to B/I_n$ satisfying
\[
\pi_nS_\eps=T,\qquad \|S_\eps\|\leqslant1+\eps.
\]
The stage $n$ is independent of $\eps$.
\end{definition}
The datum here is the prescribed tower $(I_n)$. The condition that an ideal admit some such exhaustion would impose no restriction: the constant tower $I_n=I$ is always available. Positive homogeneity permits us to replace a contraction $T$ by an arbitrary bounded lattice homomorphism, with the bound $(1+\eps)\|T\|$ on the lift. The zero map is lifted by zero.

\begin{definition}\label{pre:proj}
For $\lambda>1$, a Banach lattice $A$ is \emph{$\lambda$-projective} if, for every Banach lattice $B$, every closed ideal $I\subseteq B$, and every lattice homomorphism $T:A\to B/I$, there is a lattice homomorphism $S:A\to B$ such that
\[
QS=T,\qquad \|S\|\leqslant\lambda\|T\|,
\]
where $Q:B\to B/I$ is the quotient map. It is \emph{$1^+$-projective} if it is $\lambda$-projective for every $\lambda>1$, and \emph{$\infty$-projective} if it is $\lambda$-projective for some $\lambda>1$. Appending \emph{relative to separable ambient lattices} restricts $B$ to be separable.
\end{definition}
Relative $1^+$-projectivity implies semiprojectivity: lift to $B$ and then compose with $B\to B/I_1$. Thus the stage may always be chosen to be one in this case. The term projective in \cite{dpw} means $1^+$-projective in our terminology.

A positive element $u$ of a Banach lattice $E$ is a \emph{topological order unit} if its generated closed ideal is $E$. Every separable Banach lattice has such an element: if $(x_k)$ is dense, one may take
\[
u=\sum_{k=1}^\infty\frac{2^{-k}|x_k|}{1+\|x_k\|}.
\]
Every disjoint family of nonzero positive elements in a separable Banach lattice is countable. Indeed, the members of norm at least $1/m$ form a $1/m$-separated set, for each $m\in\N$.

\begin{proposition}\label{pre:retract}
Let $i:A\to P$ and $r:P\to A$ be contractive lattice homomorphisms with $ri=\id_A$. If $P$ is semiprojective, then so is $A$.
\end{proposition}
\begin{proof}
For a contraction $T:A\to B/I$, apply semiprojectivity of $P$ to $Tr$. If $S_\eps:P\to B/I_n$ is the resulting lift, then $S_\eps i$ lifts $T$ and has norm at most $1+\eps$. The same stage works for every $\eps$.
\end{proof}

\begin{remark}
Semiprojectivity has a formal resemblance to finite presentability: a morphism into a limit factors through a finite stage. Finite presentability, however, requires preservation of all filtered colimits by a hom-functor \cite{adamek-rosicky}, whereas our definition tests specified sequential quotient limits and permits a norm increase in the factorisation. Thus the resemblance should not be read as a categorical equivalence. In an algebraic variety, finite presentations require finitely many generators and finitely many relations; a quotient of an arbitrary infinitely generated free algebra by finitely many relations need not be finitely presentable.
\end{remark}

\section{Separable reduction and bounded lifting}\label{lift:section}

The separability restriction in our definition is significant for nonseparable
domains. For separable domains it can be removed, both for ordinary quotient
lifting and for the prescribed sequences of ideals occurring in semiprojectivity.
We give the norm-preserving reductions first.

\begin{lemma}\label{lift:separable-quotient}
Let $X$ be a Banach lattice, let $J\subseteq X$ be a closed ideal, and let
$Y\subseteq X/J$ be a separable closed sublattice. There is a separable closed
sublattice $X_0\subseteq X$ such that the quotient map $Q\colon X\to X/J$
induces an isometric lattice isomorphism
\[
X_0/(X_0\cap J)\longrightarrow Y.
\]
\end{lemma}

\begin{proof}
Choose a countable dense subset $D$ of the closed unit ball of $Y$. For
$d\in D$ and $m\in\N$, choose $x_{d,m}\in X$ with
\[
Qx_{d,m}=d,\qquad \|x_{d,m}\|\leqslant1+1/m,
\]
and let $X_0$ be the closed sublattice generated by these elements. It is
separable, $Q(X_0)\subseteq Y$, and
\[
\overline{Q(B_{X_0})}=B_Y.
\]
Indeed, $x_{d,m}/(1+1/m)\in B_{X_0}$ and its image converges to $d$.

For completeness, this density gives the exact quotient norm. Fix $y\in Y$
and $0<\delta<1$. Starting with $r_0=y$, choose recursively $x_k\in X_0$
such that
\[
\|x_k\|\leqslant\|r_{k-1}\|,\qquad
r_k=r_{k-1}-Qx_k,\qquad
\|r_k\|\leqslant\delta\|r_{k-1}\|.
\]
Then $x=\sum_{k\geqslant1}x_k$ converges in $X_0$,
$Qx=y$, and $\|x\|\leqslant\|y\|/(1-\delta)$. Letting
$\delta\downarrow0$ proves that the quotient norm equals $\|y\|$.
The asserted lattice identification follows because $Q$ is a lattice
homomorphism with kernel $X_0\cap J$.
\end{proof}

\begin{proposition}\label{lift:reduction}
For a separable Banach lattice $A$, testing $1^+$-projectivity or
$\lambda$-projectivity relative to separable ambient lattices is equivalent
to testing the corresponding property for arbitrary ambient Banach lattices.
The same assertion holds for semiprojectivity, with its requirement that the
lifting stage be independent of $\eps$.
\end{proposition}

\begin{proof}
For ordinary quotient lifting, apply Lemma~\ref{lift:separable-quotient} to
$Y=\overline{T(A)}\subseteq X/J$. The map $T$ factors through
$X_0/(X_0\cap J)$ without changing its norm. A lift into $X_0$, followed
by the inclusion into $X$, retains both its norm and the required identity.

For the assertion about semiprojectivity, let $B$ be arbitrary, let
$I=\overline{\bigcup_n I_n}$, and let $\phi\colon A\to B/I$ be contractive.
Choose a separable closed sublattice $B_1\subseteq B$ containing lifts of
$\phi(a_k)$ for a dense sequence $(a_k)$ in $A$. Inductively, choose a
countable dense subset $D_k\subseteq B_k$. For each $d\in D_k$ and
$m\in\N$, choose $z_{d,m}\in\bigcup_n I_n$ satisfying
\[
\|d-z_{d,m}\|\leqslant\operatorname{dist}(d,I)+1/m,
\]
and let $B_{k+1}$ be the closed sublattice generated by $B_k$ and all
these elements. Set
\[
B_0=\overline{\bigcup_k B_k},\qquad
K_n=B_0\cap I_n,\qquad K=\overline{\bigcup_n K_n}.
\]
Each $K_n$ is a closed ideal in the separable lattice $B_0$. The choices
above give $\operatorname{dist}(d,K)=\operatorname{dist}(d,I)$ on
$\bigcup_kD_k$, and hence on all of $B_0$, by continuity of the distance
functions. Consequently,
\[
K=B_0\cap I,
\]
and the natural map $B_0/K\to B/I$ is an isometric lattice embedding.
Its image is closed and contains every $\phi(a_k)$, so $\phi$ factors
contractively through $B_0/K$. Semiprojectivity applied to $(B_0,K_n)$
provides a fixed stage and lifts with norm at most $1+\eps$. Composing
with the contractive maps $B_0/K_n\to B/I_n$ proves the assertion.
\end{proof}

The next construction converts semiprojectivity into exact lifting through
an arbitrary quotient of a separable lattice. Its norm estimate depends on
the stage supplied by semiprojectivity; this distinction will be retained
throughout.

\begin{proposition}\label{lift:path}
Let $A$ be semiprojective, let $X$ be a separable Banach lattice, and let
$J\subseteq X$ be a closed ideal. Every bounded lattice homomorphism
$T\colon A\to X/J$ has a bounded lattice homomorphism lift into $X$.
\end{proposition}

\begin{proof}
By scaling, suppose that $T$ is contractive. Write $Q\colon X\to X/J$ for
the quotient map and consider the closed sublattice
\[
B=\{f\in C([0,1],X):f(0)=0,\ Qf(t)=tQf(1)\ (0\leqslant t\leqslant1)\}
\]
with the supremum norm. Stability under modulus follows from
$Q|f(t)|=|Qf(t)|=tQ|f(1)|$. The lattice $B$ is separable.

The map $\Pi\colon B\to X/J$, defined by $\Pi(f)=Qf(1)$, is a
contractive lattice surjection. It induces an isometric lattice
identification $B/I=X/J$, where
\[
I=\{f\in C([0,1],J):f(0)=0\}.
\]
Indeed, if $Qx=y$, then $f_x(t)=tx$ belongs to $B$,
$\Pi(f_x)=y$, and $\|f_x\|=\|x\|$; taking the infimum over such $x$
proves equality of the quotient norms.

For $n\in\N$, set
\[
I_n=\{f\in I:f|_{[0,1/n]}=0\}.
\]
These are increasing closed ideals in $B$, with dense union in $I$.
To see density, for $n\geqslant2$ choose a continuous scalar cutoff
$u_n$ taking values in $[0,1]$, equal to zero on $[0,1/n]$ and to one
on $[2/n,1]$. For $f\in I$, the elements $u_nf\in I_n$ converge
uniformly to $f$, since $f(0)=0$.

Apply semiprojectivity to $T$, viewed as a map into $B/I$. There is a
stage $n$ and, for every $\eps>0$, a lift
$\psi_\eps\colon A\to B/I_n$ with $\|\psi_\eps\|\leqslant1+\eps$.
The map
\[
E_n\colon B/I_n\to X,\qquad E_n(f+I_n)=nf(1/n),
\]
is a well-defined lattice homomorphism of norm at most $n$, and
\[
QE_n(f+I_n)=Qf(1)=\Pi(f).
\]
Thus $E_n\psi_\eps$ lifts $T$ and has norm at most $n(1+\eps)$.
In particular, a bounded lift exists.
\end{proof}

Although the preceding construction does not give norm control close to
one, a direct-sum argument makes its bounded lifting constant uniform.

\begin{theorem}\label{lift:lambda}
Every semiprojective Banach lattice is $\lambda$-projective relative to
separable ambient lattices for some finite $\lambda>1$. In particular,
every separable semiprojective Banach lattice is $\lambda$-projective
for some finite $\lambda>1$.
\end{theorem}

\begin{proof}
Suppose that no uniform constant exists. For each $m\in\N$, choose a
separable Banach lattice $X_m$, a closed ideal $J_m\subseteq X_m$, and
a contraction $T_m\colon A\to X_m/J_m$ having no lattice lift of norm
at most $m^3$. Set
\[
X=\Bigl(\bigoplus_mX_m\Bigr)_{c_0},\qquad
J=\Bigl(\bigoplus_mJ_m\Bigr)_{c_0}.
\]
The quotient $X/J$ is isometrically lattice isomorphic to
$\bigl(\bigoplus_mX_m/J_m\bigr)_{c_0}$: coordinatewise quotient
maps are contractive, and coordinatewise lifts with arbitrarily small
relative norm loss give the reverse quotient-norm inequality and
surjectivity. Hence
\[
T(a)=(m^{-1}T_m(a))_{m\geqslant1}
\]
defines a contractive lattice homomorphism $A\to X/J$. Since $X$ is
separable, Proposition~\ref{lift:path} gives a bounded lattice lift
$S\colon A\to X$. If $p_m\colon X\to X_m$ is the coordinate
projection, then $mp_mS$ lifts $T_m$ and has norm at most $m\|S\|$.
For sufficiently large $m$, this contradicts its defining property.
The final assertion follows from Proposition~\ref{lift:reduction}.
\end{proof}

\begin{proposition}\label{lift:free-retract}
Every separable semiprojective Banach lattice is a bounded lattice
retract of $\FBL(\N)$.
\end{proposition}

\begin{proof}
Choose a sequence $(a_n)$ dense in the closed unit ball of $A$.
The universal property of the free Banach lattice gives a contractive
lattice homomorphism $q\colon\FBL(\N)\to A$ sending its generators
to $a_n$. Since its image of the unit ball is dense in $B_A$, the
geometric-series argument in Lemma~\ref{lift:separable-quotient}
shows that $q$ is onto and induces an isometric lattice isomorphism
$\FBL(\N)/\ker q\to A$. The free lattice $\FBL(\N)$ is separable
\cite{dpw}. Applying Proposition~\ref{lift:path} to $\id_A$, through
this quotient identification, gives a bounded lattice homomorphism
$s\colon A\to\FBL(\N)$ with $qs=\id_A$.
\end{proof}

\begin{corollary}\label{lift:characters}
The bounded real-valued lattice homomorphisms on a separable
semiprojective Banach lattice separate its points.
\end{corollary}

\begin{proof}
In the concrete realisation of $\FBL(\N)$ as a lattice of functions
on $[-1,1]^{\N}$, point evaluations are bounded lattice homomorphisms
and separate points \cite{dpw}. Compose them with the injective map
$s$ from Proposition~\ref{lift:free-retract}.
\end{proof}

\begin{remark}\label{lift:norm-gap}
Theorem~\ref{lift:lambda} and Proposition~\ref{lift:free-retract} give
bounded lifting and a bounded lattice retraction. They do not show
that the lifting constant, or the norm of the retraction, can be
chosen arbitrarily close to one. Thus they leave open whether every
separable semiprojective Banach lattice is $1^+$-projective.
\end{remark}

\section{Disjoint ideals and direct sums}\label{sums:section}

The separability of the ambient lattice allows us to lift disjoint ranges
into disjoint closed ideals before solving the individual lifting problems.
This observation removes the order-unit hypothesis from the semiprojectivity
assertion for arbitrary $\ell_1$-sums and also gives finite-sum permanence.
We first record the ideal facts needed to retain the quotient norms.

\begin{lemma}\label{sums:ideal-quotient}
Let $H$ and $J$ be closed ideals in a Banach lattice $B$, and let
$Q\colon B\to B/J$ be the quotient map. Then $Q(H)$ is a closed ideal and
the natural map
\[
H/(H\cap J)\longrightarrow Q(H)
\]
is an isometric lattice isomorphism. If $H$ is the closed ideal generated
by $u\in B_+$, then $Q(H)$ is the closed ideal generated by $Qu$.
\end{lemma}

\begin{proof}
For $x\in H$ and $j\in J$, put
\[
j'=(j^+\wedge x^+)-(j^-\wedge x^-).
\]
Then $j'\in H\cap J$, and the lattice inequality
$|x-j'|\leqslant |x-j|$ gives
\[
\inf_{j\in J}\|x-j\|
=\inf_{j\in H\cap J}\|x-j\|.
\]
Thus the natural injective lattice homomorphism from $H/(H\cap J)$ into
$B/J$ is isometric, so its range is closed.

To see that the range is an ideal, suppose that $x\in H$ and
$z\in B/J$ satisfy $|z|\leqslant |Qx|$. Choose $w\in B$ with $Qw=z$ and
set
\[
h=(w\vee(-|x|))\wedge |x|.
\]
Then $h\in H$ and $Qh=z$. Finally, if $H$ is generated by $u$, every
element of $H$ is a norm limit of elements whose moduli are bounded by
positive multiples of $u$. Their images belong to the principal ideal
generated by $Qu$. Conversely, $Q(H)$ is a closed ideal containing $Qu$.
These two inclusions prove the last assertion.
\end{proof}

\begin{lemma}\label{sums:ideal-intersections}
Let $(I_n)_{n\in\N}$ be an increasing sequence of closed ideals in a
Banach lattice $B$, let $I=\overline{\bigcup_n I_n}$, and let $H$ be a
closed ideal in $B$. Then
\[
I\cap H=\overline{\bigcup_{n\in\N}(I_n\cap H)}.
\]
\end{lemma}

\begin{proof}
The right-hand side is contained in $I\cap H$. For the reverse inclusion,
let $x\in(I\cap H)_+$ and choose $y_k\in I_{n_k}$ with $y_k\to x$.
Then $z_k=|y_k|\wedge x$ belongs to $I_{n_k}\cap H$ and
\[
\|x-z_k\|\leqslant\|x-|y_k|\|\leqslant\|x-y_k\|\longrightarrow0.
\]
Applying this argument to the positive and negative parts proves the
assertion for every $x\in I\cap H$.
\end{proof}

\begin{lemma}\label{sums:countable-support}
Let $(E_\gamma)_{\gamma\in\Gamma}$ be Banach lattices, let $F$ be a
separable Banach lattice, and let
\[
T\colon\Bigl(\bigoplus_{\gamma\in\Gamma}E_\gamma\Bigr)_{\ell_1}
\longrightarrow F
\]
be a lattice homomorphism. Then $T$ vanishes on all but countably many
coordinates.
\end{lemma}

\begin{proof}
For each coordinate on which $T$ is nonzero, choose a positive vector
$x_\gamma$ in that coordinate with $Tx_\gamma\ne0$. The vectors
$Tx_\gamma$ are pairwise disjoint. Every disjoint family of nonzero
positive elements in a separable Banach lattice is countable: for each
$m\in\N$, the members of norm at least $1/m$ are $1/m$-separated, since
disjoint positive vectors $u,v$ satisfy
\[
\|u-v\|=\|u+v\|\geqslant\|u\|.
\]
Each such separated family is countable, and their countable union
contains the whole family. This proves the claim.
\end{proof}

No topological order unit in the domain is needed for this support
reduction. The order units used below belong instead to closed ideals
of the separable target. Recall that every separable Banach lattice has
a topological order unit: if $(v_k)$ is dense in its positive cone, then
\[
v=\sum_{k=1}^{\infty}2^{-k}\frac{v_k}{1+\|v_k\|}
\]
generates a closed ideal containing every $v_k$.

\begin{theorem}\label{sums:arbitrary-sums}
Let $(P_\gamma)_{\gamma\in\Gamma}$ be a family of Banach lattices that
are $1^+$-projective relative to separable ambient lattices. Then
\[
P=\Bigl(\bigoplus_{\gamma\in\Gamma}P_\gamma\Bigr)_{\ell_1}
\]
has the same property. In particular, an arbitrary $\ell_1$-sum of
$1^+$-projective Banach lattices is semiprojective.
\end{theorem}

\begin{proof}
Let $B$ be separable, let $I\subseteq B$ be a closed ideal, and write
$Q\colon B\to F=B/I$. Fix a lattice homomorphism $T\colon P\to F$ and
$\eps>0$. The zero map has the zero lift, so assume that $T\ne0$.
By Lemma~\ref{sums:countable-support}, the set $\Gamma_0$ of active
coordinates is countable.

For $\gamma\in\Gamma_0$, let $T_\gamma$ be the restriction of $T$ to
$P_\gamma$ and let $F_\gamma$ be the closed ideal of $F$ generated by
$T_\gamma(P_\gamma)$. The ideals $F_\gamma$ are pairwise disjoint:
disjointness of the coordinate ranges passes to their generated ideals
and then to their norm closures. Each $F_\gamma$ is separable, so choose
a topological order unit $v_\gamma\in(F_\gamma)_+$.

By the countable disjoint-lifting theorem
\cite[Theorem~9.1]{dpw}, there are pairwise disjoint
$u_\gamma\in B_+$ with $Qu_\gamma=v_\gamma$. Let $B_\gamma$ be the
closed ideal generated by $u_\gamma$. These ideals are pairwise disjoint,
and Lemma~\ref{sums:ideal-quotient} gives
\[
B_\gamma/(B_\gamma\cap I)\cong F_\gamma
\]
isometrically. Since $B_\gamma$ is separable, the assumed relative
projectivity supplies a lattice homomorphism
$\widehat T_\gamma\colon P_\gamma\to B_\gamma$ satisfying
\[
Q\widehat T_\gamma=T_\gamma,
\qquad
\|\widehat T_\gamma\|\leqslant(1+\eps)\|T_\gamma\|
\leqslant(1+\eps)\|T\|.
\]
Define
\[
\widehat T((x_\gamma)_{\gamma\in\Gamma})
=\sum_{\gamma\in\Gamma_0}\widehat T_\gamma x_\gamma.
\]
The series converges absolutely in norm, and
$\|\widehat T\|\leqslant(1+\eps)\|T\|$. Disjointness of the ideals
$B_\gamma$ shows first on finitely supported vectors, and then by
continuity on all of $P$, that $\widehat T$ is a lattice homomorphism.
Moreover, $Q\widehat T=T$. This proves relative projectivity, which
implies semiprojectivity by composing the lift with any stage quotient.
\end{proof}

\begin{theorem}\label{sums:finite-sums}
A finite $\ell_1$-sum of Banach lattices is semiprojective if and only
if each summand is semiprojective.
\end{theorem}

\begin{proof}
Suppose that $A_1,\ldots,A_m$ are semiprojective. Let $B$ be separable,
let $I=\overline{\bigcup_n I_n}$ for increasing closed ideals $I_n$,
and let
\[
\phi\colon\bigoplus_{j=1}^{m}A_j\longrightarrow B/I
\]
be a contractive lattice homomorphism, where the domain carries the
$\ell_1$-norm. Write $Q\colon B\to B/I$ and let $\phi_j$ be the
coordinate restrictions. Discarding zero restrictions, if necessary,
the disjoint-ideal construction in the preceding proof gives pairwise
disjoint closed ideals $B_j\subseteq B$ such that $Q(B_j)$ contains
$\phi_j(A_j)$.

Put $J_j=I\cap B_j$ and $J_{j,n}=I_n\cap B_j$. By
Lemmas~\ref{sums:ideal-quotient} and \ref{sums:ideal-intersections},
\[
B_j/J_j\cong Q(B_j),\qquad
J_j=\overline{\bigcup_nJ_{j,n}},
\]
with the first identification isometric. We may therefore regard
$\phi_j$ as a contractive map into $B_j/J_j$. Semiprojectivity of $A_j$
provides an index $n_j$ such that, for every $\eps>0$, this map lifts to
$B_j/J_{j,n_j}$ with norm at most $1+\eps$.

Let $N=\max_j n_j$. For a fixed $\eps>0$, compose these lifts with the
canonical quotient maps to $B_j/J_{j,N}$ and the isometric lattice
embeddings
\[
B_j/(B_j\cap I_N)\longrightarrow B/I_N.
\]
Call the resulting maps $\psi_j$. Their ranges are pairwise disjoint,
because they lie in the images of the disjoint ideals $B_j$. Hence
\[
\psi(a_1,\ldots,a_m)=\sum_{j=1}^{m}\psi_j(a_j)
\]
is a lattice homomorphism lifting $\phi$, and
\[
\|\psi(a_1,\ldots,a_m)\|
\leqslant(1+\eps)\sum_{j=1}^{m}\|a_j\|.
\]
The index $N$ is independent of $\eps$, as required. Zero restrictions
are restored by taking zero lifts. If every restriction is zero, the
zero lift at the first stage suffices.

Conversely, each summand is a contractive lattice retract of the sum.
Composing a lifting problem with the coordinate projection, and its
solution with the coordinate inclusion, proves semiprojectivity of
that summand.
\end{proof}

\begin{corollary}\label{sums:order-unit-sums}
Let $(P_\gamma)_{\gamma\in\Gamma}$ be $1^+$-projective Banach lattices.
Their $\ell_1$-sum is semiprojective. If $\Gamma$ is countable and each
$P_\gamma$ has a topological order unit, then the sum is
$1^+$-projective.
\end{corollary}

\begin{proof}
The first assertion follows from Theorem~\ref{sums:arbitrary-sums}.
The second is \cite[Theorem~11.8]{dpw}.
\end{proof}

\begin{corollary}\label{sums:ell-one}
For every set $\Gamma$, the Banach lattice $\ell_1(\Gamma)$ is
$1^+$-projective relative to separable ambient lattices, and hence
semiprojective. It is $1^+$-projective if and only if $\Gamma$ is
countable.
\end{corollary}

\begin{proof}
Apply Theorem~\ref{sums:arbitrary-sums} to the one-dimensional
projective summands $\R$; see \cite[Corollary~10.6]{dpw}.
For countable $\Gamma$, Corollary~\ref{sums:order-unit-sums} applies.
For uncountable $\Gamma$, the coordinate vectors form an uncountable
disjoint family, whereas every disjoint family of nonzero positive
elements in a projective Banach lattice is countable
\cite[Corollary~10.5]{dpw}.
\end{proof}

\begin{corollary}\label{sums:anr-sums}
If $(K_\gamma)_{\gamma\in\Gamma}$ is a family of compact metric ANRs,
then $\bigl(\bigoplus_\gamma C(K_\gamma)\bigr)_{\ell_1}$ is
semiprojective. For countable $\Gamma$, it is $1^+$-projective.
\end{corollary}

\begin{proof}
Each $C(K_\gamma)$ is $1^+$-projective by
\cite[Theorem~1.4]{avilaes2020} and has the order unit
$\mathbf1_{K_\gamma}$. Apply Corollary~\ref{sums:order-unit-sums}.
\end{proof}

\begin{remark}\label{sums:free-example}
The order-unit hypothesis cannot be omitted from the unrestricted
projectivity assertion, even for two summands. If $\Gamma$ is
uncountable, de Pagter and Wickstead show that
$\R\oplus\FBL(\Gamma)$ is not projective under any Banach lattice norm
\cite[Example~11.9]{dpw}. Nevertheless,
$\R\oplus_1\FBL(\Gamma)$ is $1^+$-projective relative to separable
ambient lattices by Theorem~\ref{sums:arbitrary-sums}, since free Banach
lattices are $1^+$-projective. It is therefore semiprojective.

Both this example and $\ell_1(\Gamma)$ distinguish the paper's
separable-ambient notion from unrestricted projectivity. They do not
distinguish semiprojectivity from projectivity relative to separable
ambient lattices: the latter, stronger lifting property holds for
these examples as well.
\end{remark}

\section{Obstructions and classical examples}\label{sec:obstructions}

Theorem~\ref{lift:lambda} allows us to organise the negative
examples around two necessary conditions: bounded sequences which can be mapped
onto the basis of $c_0$ must exhibit $\ell_1$ behaviour, and real-valued lattice
homomorphisms must separate points.  The first condition follows from the work
of Avil\'es, Mart\'{\i}nez-Cervantes and Rodr\'{\i}guez Abell\'an; the second
has particularly strong consequences when the norm is order continuous.

\begin{proposition}\label{obs:sequence}
Let $E$ be a separable semiprojective Banach lattice, and let $(x_n)_{n\in\N}$
be a bounded sequence in $E$.  Suppose that there is a bounded lattice
homomorphism $T:E\to c_0$ such that $T(x_n)=e_n$ for every $n\in\N$.
Then $(x_n)$ has a subsequence equivalent to the canonical basis of $\ell_1$.
\end{proposition}

\begin{proof}
By Theorem~\ref{lift:lambda}, $E$ is $\infty$-projective.  The conclusion is
therefore precisely \cite[Theorem~1.2]{avilaes2020}.
\end{proof}

This proposition replaces separate lifting constructions for the classical
sequence lattices.  Its conclusion concerns arbitrary scalar combinations,
rather than only the growth of positive partial sums, and consequently retains
more information than is needed for the following examples.

\begin{corollary}\label{obs:sequence-lattices}
Let $\Gamma$ be an infinite set.  Then $c_0(\Gamma)$ is not semiprojective.
For every $1<p<\infty$, the Banach lattice $\ell_p(\Gamma)$ is not
semiprojective either.
\end{corollary}

\begin{proof}
First consider $\Gamma=\N$.  For $c_0$, use the identity map and its canonical
basis in Proposition~\ref{obs:sequence}.  For $\ell_p$, use the formal
inclusion $\ell_p\to c_0$ and again take the canonical basis.  Both maps are
contractive lattice homomorphisms.  No subsequence of either basis is
equivalent to the basis of $\ell_1$: for distinct indices $n_1,\ldots,n_m$,
\[
 \Bigl\|\sum_{j=1}^m e_{n_j}\Bigr\|_{c_0}=1,
 \qquad
 \Bigl\|\sum_{j=1}^m e_{n_j}\Bigr\|_{\ell_p}=m^{1/p},
\]
whereas an $\ell_1$ lower estimate would make these norms at least $cm$ for
some fixed $c>0$.  This proves the countable cases.

For arbitrary infinite $\Gamma$, choose a countably infinite subset
$\Delta\subseteq\Gamma$.  Coordinate restriction and extension by zero are
contractive lattice homomorphisms whose composition on $c_0(\Delta)$,
respectively $\ell_p(\Delta)$, is the identity.  Thus these countable sequence
lattices are contractive lattice retracts of the corresponding spaces indexed
by $\Gamma$.  Proposition~\ref{pre:retract} finishes the proof.
\end{proof}

\begin{proposition}\label{obs:finite-dimensional}
Every finite-dimensional Banach lattice is $1^+$-projective, and hence
semiprojective.
\end{proposition}

\begin{proof}
The projectivity assertion is \cite[Theorem~11.1]{dpw}.  A lift to the ambient
Banach lattice can be composed with any of the finite-stage quotient maps,
which gives semiprojectivity with the same norm bound.
\end{proof}

We next exploit separation by lattice homomorphisms.  Recall that a nonzero
positive element $a$ is an \emph{atom} if every $x$ with $0\leqslant x\leqslant a$
is a scalar multiple of $a$.  A Banach lattice is \emph{atomic} if every
nonzero positive element dominates an atom.  For a normalised atom $a$, its
band projection has the form
\[
 P_a x=h_a(x)a,
\]
where $h_a:E\to\R$ is a norm-one lattice homomorphism, called its coordinate
functional.  Here and below we use the standard facts that an order-continuous
Banach lattice is Dedekind complete and that its norm-closed ideals are
projection bands; see \cite[Section~2.4]{meyer-nieberg}.

\begin{lemma}\label{obs:characters-atoms}
Let $E$ be a Banach lattice with order-continuous norm.  Every nonzero member
of $\Hom(E,\R)$ is a positive multiple of the coordinate functional of an
atom.  Consequently, if $\Hom(E,\R)$ separates points of $E$, then $E$ is
atomic.  If $E$ is atomless, then $\Hom(E,\R)=\{0\}$.
\end{lemma}

\begin{proof}
Let $h:E\to\R$ be a nonzero lattice homomorphism.  Its kernel $N$ is a
norm-closed ideal, so $E=N\oplus N^d$ is a decomposition into complementary
bands.  The restriction of $h$ to $N^d$ is injective.  Since its range is
contained in $\R$, the band $N^d$ is one-dimensional.  It is therefore spanned
by a normalised atom $a$, and
\[
 h(x)=h(a)h_a(x)\qquad(x\in E),
\]
where $h(a)>0$.

Now suppose that characters separate points and take $x>0$.  There is a
character $h$ with $h(x)>0$.  For the corresponding atom $a$, the element
$P_a x=h_a(x)a$ is a nonzero atom dominated by $x$.  Thus $E$ is atomic.
The assertion for atomless lattices follows from the first part.
\end{proof}

The last assertion also follows from the relation between characters and atoms
of the dual, as recalled before \cite[Example~3.3]{dantas2022}.  The band
argument above records additionally how the character is located inside $E$.

\begin{theorem}\label{obs:order-continuous}
Let $E$ be a separable semiprojective Banach lattice with order-continuous
norm.  Then $E$ is atomic.  Moreover, every sequence of distinct normalised
atoms of $E$ has a subsequence equivalent to the canonical basis of $\ell_1$.
In particular, if $E$ is infinite-dimensional, it contains a closed sublattice
lattice isomorphic to $\ell_1$.
\end{theorem}

\begin{proof}
Characters separate points of $E$ by Corollary~\ref{lift:characters}, so
Lemma~\ref{obs:characters-atoms} gives atomicity.  Let $(a_n)$ be distinct
normalised atoms, and write $h_n=h_{a_n}$.  The map
\[
 T:E\longrightarrow c_0,\qquad T(x)=(h_n(x))_{n\in\N},
\]
is well-defined and contractive.  Indeed, the positive vectors
$|h_n(x)|a_n=P_{a_n}|x|$ are pairwise disjoint and dominated by $|x|$.
Order continuity implies that their norms tend to zero, hence $h_n(x)\to0$.
Also $|h_n(x)|\leqslant\|x\|$ for every $n$.  Each coordinate of $T$ is a
lattice homomorphism, so $T$ is a lattice homomorphism, and $T(a_n)=e_n$.
Proposition~\ref{obs:sequence} now gives the asserted subsequence.

If $E$ is infinite-dimensional and atomic, it has infinitely many distinct
atoms.  The closed linear span of an $\ell_1$-equivalent subsequence of
normalised atoms is a sublattice: disjointness gives
\[
 \Bigl|\sum_{j=1}^m t_j a_{n_j}\Bigr|
 =\sum_{j=1}^m |t_j|a_{n_j},
\]
and this identity extends by continuity.  The resulting isomorphism from
$\ell_1$ is therefore a lattice isomorphism.
\end{proof}

\begin{corollary}\label{obs:reflexive}
A separable reflexive Banach lattice is semiprojective if and only if it is
finite-dimensional.
\end{corollary}

\begin{proof}
Reflexive Banach lattices have order-continuous norm.  An infinite-dimensional
separable semiprojective one would contain $\ell_1$ by
Theorem~\ref{obs:order-continuous}, which is impossible in a reflexive Banach
space.  The converse is Proposition~\ref{obs:finite-dimensional}.
\end{proof}

\begin{corollary}\label{obs:atomless}
Every nonzero separable atomless Banach lattice with order-continuous norm is
not semiprojective.  In particular, $L_p([0,1])$ is not semiprojective for
every $1\leqslant p<\infty$.
\end{corollary}

\begin{proof}
Lemma~\ref{obs:characters-atoms} gives $\Hom(E,\R)=\{0\}$, contrary to
Corollary~\ref{lift:characters}.  The usual lattices $L_p([0,1])$ satisfy
all the stated hypotheses.
\end{proof}

For Orlicz hearts it is useful to give a direct argument, which applies to
arbitrary atomless measure spaces and makes no assumption that the heart
equals the full Orlicz space.  In the next statement an Orlicz function means
a finite-valued, convex, nondecreasing function
$\Phi:[0,\infty)\to[0,\infty)$ with $\Phi(0)=0$ which is not identically
zero.  This convention allows $\Phi$ to vanish on an initial interval.

\begin{proposition}\label{obs:orlicz}
Let $(\Omega,\Sigma,\mu)$ be an atomless measure space and let $\Phi$ be an
Orlicz function.  Equip the Orlicz heart
\[
 H^\Phi(\mu)=
 \left\{f\in L^0(\mu):
   \int_\Omega\Phi(|f|/\lambda)\,d\mu<\infty
   \text{ for every }\lambda>0\right\}
\]
with the Luxemburg norm.  Then $\Hom(H^\Phi(\mu),\R)=\{0\}$.
Consequently, a nonzero separable $H^\Phi(\mu)$ is not semiprojective.
The same conclusions hold for $L^\Phi(\mu)$ whenever
$L^\Phi(\mu)=H^\Phi(\mu)$, with nonzero separability assumed for the
semiprojectivity conclusion.
\end{proposition}

\begin{proof}
Suppose that $\varphi:H^\Phi(\mu)\to\R$ is a nonzero lattice homomorphism.
Choose $f\geqslant0$ in the heart with $\varphi(f)>0$, and fix $\eps>0$.
The function $g=\Phi(f/\eps)$ is integrable, so
\[
 \nu(A)=\int_A g\,d\mu\qquad(A\in\Sigma)
\]
defines a finite atomless measure.  For completeness, no
$\sigma$-finiteness assumption on $\mu$ is needed here: the sets
$\{g\geqslant1/r\}$, $r\in\N$, have finite $\mu$-measure and cover the
support of $g$ up to a null set.  On each of these sets the measure with
density $g$ is atomless because $\mu$ is atomless.

Choose $m\in\N$ with $\nu(\Omega)/m\leqslant1$.  Divisibility of finite
atomless measures gives a measurable partition $A_1,\ldots,A_m$ of
$\Omega$ with $\nu(A_j)=\nu(\Omega)/m$.  Put $f_j=f\mathbf1_{A_j}$.
These are disjoint positive elements of the heart, their sum is $f$, and
\[
 \int_\Omega\Phi(f_j/\eps)\,d\mu
 =\nu(A_j)\leqslant1.
\]
Thus $\|f_j\|_\Phi\leqslant\eps$ for every $j$.

The nonnegative numbers $\varphi(f_j)$ are pairwise disjoint, so at most one
is nonzero.  Their sum is $\varphi(f)>0$, and therefore one of them equals
$\varphi(f)$.  It follows that
\[
 0<\varphi(f)\leqslant\|\varphi\|\eps.
\]
Since $\eps$ was arbitrary, this is a contradiction.  Non-semiprojectivity
now follows from Corollary~\ref{lift:characters}.  The assertions for the
full Orlicz space follow under the expressly stated equality with its heart.
\end{proof}

There is a further obstruction for lattices freely generated by Banach
spaces.  It concerns the linear structure of the generating space, even
though semiprojectivity is a property of the resulting lattice.

\begin{corollary}\label{obs:free-schur}
Let $X$ be a separable Banach space.  If $\FBL[X]$ is semiprojective, then
$X$ has the Schur property.  Equivalently, $\FBL[X]$ is not semiprojective
whenever $X$ is separable and admits a weakly null sequence which is not norm
null.
\end{corollary}

\begin{proof}
The lattice $\FBL[X]$ is separable: the images of a countable dense subset
of $X$ generate a dense sublattice, by the defining construction
\cite{aviles-fblE}.  Semiprojectivity therefore implies
$\infty$-projectivity by Theorem~\ref{lift:lambda}.  The conclusion follows
from \cite[Theorem~1.3]{avilaes2020}.
\end{proof}

In particular, this applies to $\FBL[c_0]$, to $\FBL[\ell_p]$ for
$1<p<\infty$, and to the free Banach lattice generated by any
infinite-dimensional separable reflexive Banach space.  These conclusions
are necessary-condition arguments; they do not assert that the Schur
property alone is sufficient for semiprojectivity of $\FBL[X]$.

\section{Spaces of continuous functions}\label{sec:topology}

We first identify the semiprojective lattices among the separable spaces of
continuous functions.  For compact spaces, the positive direction is the
projectivity theorem of Avil\'es, Mart\'{\i}nez-Cervantes, and
Rodr\'{\i}guez Abell\'an~\cite[Theorem~1.4]{avilaes2020}.  The converse
uses only the existence of a bounded lift, without any estimate for its norm.
We then pass to locally compact spaces; here noncompactness itself is an
obstruction.

\begin{lemma}\label{top:weighted}
Let $K$ and $L$ be compact Hausdorff spaces and let
$T\colon C(K)\to C(L)$ be a lattice homomorphism.  Set
$h=T\mathbf 1$ and $U=\{y\in L:h(y)>0\}$.  There is a unique continuous
map $r\colon U\to K$ such that
\[
 (Tf)(y)=
 \begin{cases}
 h(y)f(r(y)),&y\in U,\\
 0,&y\in L\setminus U,
 \end{cases}
 \qquad f\in C(K).
\]
\end{lemma}

\begin{proof}
We first observe that every nonzero lattice homomorphism
$\varphi\colon C(K)\to\R$ is a positive multiple of an evaluation.
Positivity and the inequality $|f|\leqslant\|f\|\mathbf 1$ show that
$c:=\varphi(\mathbf 1)>0$.  After replacing $\varphi$ by $c^{-1}\varphi$,
we may suppose that $\varphi(\mathbf 1)=1$.
The zero sets of the elements of $\ker\varphi$ have the finite-intersection
property.  Indeed, if $f_1,\ldots,f_m\in\ker\varphi$ had no common zero,
compactness would give $\sum_{j=1}^m|f_j|\geqslant\delta\mathbf 1$ for
some $\delta>0$, contrary to
$\varphi(\sum_j|f_j|)=0$.  Choose a common zero $x\in K$.
Since $f-\varphi(f)\mathbf 1\in\ker\varphi$, we obtain
$f(x)=\varphi(f)$ for every $f\in C(K)$.  Continuous functions separate
points, so $x$ is unique.

Apply this observation to $f\mapsto(Tf)(y)$.  If $h(y)=0$, then
$|(Tf)(y)|\leqslant\|f\|h(y)=0$.  Otherwise the observation supplies a
unique point $r(y)\in K$ and gives the asserted formula.  For every
$f\in C(K)$, the function
$f\circ r=(Tf)|_U/h|_U$ is continuous.  As the topology of $K$ is
generated by its continuous real-valued functions, $r$ is continuous.
\end{proof}

If $F$ is a closed subset of a compact Hausdorff space $Y$, write
\[
 I_F=\{f\in C(Y):f|_F=0\}.
\]
Restriction induces an isometric lattice isomorphism
$C(Y)/I_F\cong C(F)$, by the Tietze extension theorem.

\begin{lemma}\label{top:tower}
Let $Y$ be a compact metric space, and let
$Y_1\supseteq Y_2\supseteq\cdots$ be compact subsets with intersection
$K$.  Then
\[
 I_K=\overline{\bigcup_{n\in\N}I_{Y_n}}.
\]
\end{lemma}

\begin{proof}
Only one inclusion needs proof.  Given $f\in I_K$, put
$\delta_n=\sup_{Y_n}|f|$, taking the supremum over the empty set to be
zero.  Compactness and the nesting of the sets $Y_n$ imply
$\delta_n\to0$: otherwise a convergent subsequence of suitable points
in the $Y_n$ would have a limit in $K$ at which $f$ does not vanish.
The functions
\[
 g_n=(f-\delta_n\mathbf 1)^+
       -(-f-\delta_n\mathbf 1)^+
\]
vanish on $Y_n$ and satisfy $\|f-g_n\|\leqslant\delta_n$.
Thus $f$ belongs to the indicated closure.
\end{proof}

\begin{theorem}\label{top:CK}
Let $K$ be a compact metric space.  The following conditions are equivalent:
\begin{enumerate}
\item $C(K)$ is semiprojective;
\item $K$ is an absolute neighbourhood retract;
\item $C(K)$ is $1^+$-projective.
\end{enumerate}
\end{theorem}

\begin{proof}
The implication (2)$\Rightarrow$(3) is
\cite[Theorem~1.4]{avilaes2020}, and (3)$\Rightarrow$(1) follows directly
from the definitions.

Suppose that $K$ is not an ANR.  By the classical characterisation of
compact metric ANRs, there is a closed embedding of $K$ into a compact
metric space $Y$ for which $K$ is not a neighbourhood retract; one may
take $Y$ to be the Hilbert cube~\cite[Chapter~IV]{borsuk}.
Fix a compatible metric $d$ on $Y$ and put
\[
 Y_n=\{y\in Y:d(y,K)\leqslant1/n\},\qquad
 W_n=\{y\in Y:d(y,K)<1/n\}.
\]
Let $B=C(Y)$, $I=I_K$, and $I_n=I_{Y_n}$.  The ambient lattice $B$ is
separable, and Lemma~\ref{top:tower} gives
$I=\overline{\bigcup_nI_n}$.  Under the canonical identifications
$B/I=C(K)$ and $B/I_n=C(Y_n)$, the map $B/I_n\to B/I$ is restriction.

If $C(K)$ were semiprojective, the identity of $C(K)$ would therefore
admit a bounded lattice lift
$T\colon C(K)\to C(Y_n)$ for some $n$; in particular,
$(Tf)|_K=f$.  Apply Lemma~\ref{top:weighted}.  Since
$h=T\mathbf 1$ satisfies $h|_K=1$, its positivity set $U\subseteq Y_n$
contains $K$.  For $x\in K$ and $f\in C(K)$,
\[
 f(x)=(Tf)(x)=f(r(x)),
\]
so $r(x)=x$.  The set $U\cap W_n$ is open in $Y$ and contains $K$,
because $U$ is relatively open in $Y_n$ and $W_n$ is open in $Y$ with
$W_n\subseteq Y_n$.  Consequently $r|_{U\cap W_n}$ is a retraction of
an open neighbourhood of $K$ onto $K$, a contradiction.
\end{proof}

The preceding proof also identifies a small family of test problems:
for any fixed Hilbert-cube embedding of $K$, it suffices to ask whether
the identity on $C(K)$ has a bounded lattice lift to one of the spaces
$C(Y_n)$ above.  In particular, the topological obstruction survives
if all norm requirements in semiprojectivity are discarded.

\begin{proposition}\label{top:c0-retract}
If $X$ is a noncompact $\sigma$-compact locally compact Hausdorff space,
then $c_0$ is a contractive lattice retract of $C_0(X)$.
\end{proposition}

\begin{proof}
Choose compact sets $K_n$ whose interiors cover $X$, with
$K_n\subseteq K_{n+1}$.  Inductively choose $x_n\in X$ and relatively
compact open neighbourhoods $U_n$ of $x_n$ such that
\[
 \overline{U_n}\cap
 \left(K_n\cup\bigcup_{j<n}\overline{U_j}\right)=\varnothing.
\]
This is possible because the set in parentheses is compact and $X$ is
not compact.  Choose $u_n\in C_c(X)$ with
$0\leqslant u_n\leqslant1$, $u_n(x_n)=1$, and
$\supp u_n\subseteq U_n$.  Define
\[
 S\colon c_0\to C_0(X),\qquad S(a)=\sum_{n=1}^\infty a_nu_n.
\]
The series converges uniformly, since the $u_n$ are disjoint and
$a_n\to0$.  Its sum belongs to $C_0(X)$, and disjointness shows that
$S$ is an isometric lattice homomorphism.

Every compact subset of $X$ is contained in some $K_n$, so the points
$x_n$ eventually leave every compact set.  It follows that
\[
 R\colon C_0(X)\to c_0,\qquad Rf=(f(x_n))_{n\in\N},
\]
is a well-defined contractive lattice homomorphism.  Our choices give
$RS=\id_{c_0}$.
\end{proof}

\begin{theorem}\label{top:C0}
Let $X$ be a locally compact second-countable Hausdorff space.  Then
$C_0(X)$ is semiprojective if and only if $X$ is compact and is an ANR.
\end{theorem}

\begin{proof}
The space $X$ is $\sigma$-compact. If it is noncompact, Proposition~\ref{top:c0-retract} and
Proposition~\ref{pre:retract} would make $c_0$ semiprojective whenever
$C_0(X)$ is semiprojective, contrary to Corollary~\ref{obs:sequence-lattices}.
If $X$ is compact, it is metrisable, and the assertion is
Theorem~\ref{top:CK}.
\end{proof}

\begin{corollary}\label{top:ideal}
Let $K$ be a compact metric space and $F\subseteq K$ a proper closed
subset.  The ideal $I_F\subseteq C(K)$ is semiprojective if and only if
$F$ is clopen and $K\setminus F$ is an ANR.
\end{corollary}

\begin{proof}
Extension by zero identifies $C_0(K\setminus F)$ isometrically with
$I_F$.  Theorem~\ref{top:C0} applies, and the open subset
$K\setminus F$ is compact precisely when $F$ is also open.
\end{proof}

For $F=K$, the ideal is zero and hence semiprojective.  Thus the
corollary describes all semiprojective closed ideals of a separable
$C(K)$-space, with this additional trivial case.

\subsection*{Ideals, quotients, pullbacks, and extensions}

The topological characterisation gives explicit failures of permanence
under ideals and quotients. We then combine it with an $\ell_\infty$-sum
obstruction to show that semiprojectivity also fails to pass to finite
$\ell_\infty$-sums, extensions, and pullbacks.

\begin{proposition}\label{top:ideal-failure}
Semiprojectivity is preserved neither by passing to closed ideals nor
by passing to quotients by closed ideals.
\end{proposition}

\begin{proof}
The lattice $C([0,1])$ is semiprojective by Theorem~\ref{top:CK}, while
its ideal $I_{\{0\}}\cong C_0((0,1])$ is not, by
Theorem~\ref{top:C0}. If $K$ is the middle-third Cantor set, then
$C([0,1])/I_K\cong C(K)$ is not semiprojective: the Cantor set is not
an ANR, since it is not locally contractible.
\end{proof}

\begin{proposition}\label{top:pullback}
Let $K\subseteq[0,1]$ be the middle-third Cantor set. The pullback
\[
 P=\{(f,g)\in C([0,1])\oplus_\infty C([0,1]):f|_K=g|_K\}
\]
is not semiprojective, although both copies of $C([0,1])$ are
semiprojective.
\end{proposition}

\begin{proof}
Let $Z$ be the quotient of the disjoint union
$[0,1]_0\sqcup[0,1]_1$ obtained by identifying the two copies of each
point of $K$, and let $q$ denote the quotient map. This is a compact
metric space. A continuous function on $Z$ is exactly a pair of
continuous functions on the two intervals that agree on $K$, so
$P$ is isometrically lattice isomorphic to $C(Z)$.

We show that $Z$ is not locally contractible at any point $q(t)$ with
$t\in K$. If $(a,b)$ is a complementary interval of $K$, then
\[
 L_{a,b}=q([a,b]_0)\cup q([a,b]_1)
\]
is a circle. Its inclusion into $Z$ is not null-homotopic. Indeed,
define $F_0,F_1\colon[0,1]\to S^1$ to be $1$ on $[0,a]$, $-1$ on
$[b,1]$, and, for $s\in[a,b]$, set
\[
 F_0(s)=\exp\left(\mathrm i\pi\frac{s-a}{b-a}\right),\qquad
 F_1(s)=\exp\left(-\mathrm i\pi\frac{s-a}{b-a}\right).
\]
The maps agree on $K$ and therefore induce a continuous map
$F\colon Z\to S^1$. Its restriction to $L_{a,b}$ is a homeomorphism,
and consequently cannot be null-homotopic.

Every neighbourhood of a point $t\in K$, relative to $[0,1]$,
contains the closure of a complementary interval. Hence every open
neighbourhood $V$ of $q(t)$ in $Z$ contains one of the circles
$L_{a,b}$: choose a sufficiently small interval about $t$ whose two
copies lie in $q^{-1}(V)$, and then choose such a complementary
interval inside it. If $V\subseteq U$ are open neighbourhoods of
$q(t)$ and the inclusion $V\hookrightarrow U$ were null-homotopic,
the inclusion of this circle into $Z$ would also be null-homotopic.
This contradicts the preceding paragraph. Thus $Z$ is not locally
contractible and therefore is not an ANR~\cite{hanner-anr}.
Theorem~\ref{top:CK} now gives the conclusion.
\end{proof}

In Proposition~\ref{top:pullback}, the base of the pullback is $C(K)$,
which is not semiprojective, and the kernel of the coordinate projection
$P\to C([0,1])$ is $I_K$, which is not semiprojective either. Notice
also that for any two compact metric ANRs $K_1$ and $K_2$, the finite
$\ell_\infty$-sum $C(K_1)\oplus_\infty C(K_2)\cong C(K_1\sqcup K_2)$ is
$1^+$-projective and semiprojective by Theorem~\ref{top:CK}. By
contrast, replacing one of the summands by $\ell_1$ produces an
obstruction to lifting in the $\ell_\infty$-norm.

\begin{theorem}\label{top:linfty-obstruction}
Let $E$ be a Banach lattice with $\Hom(E,\R)\ne\{0\}$, and let $F$ be a
Banach lattice that admits a sequence $(f_n)_{n\in\N}$ in $B_F^+$ and a
contractive lattice homomorphism $\Psi\colon F\to c_0$ such that
$(\Psi f_n)_n=1$ for every $n\in\N$. Then $E\oplus_\infty F$ is not
semiprojective.
\end{theorem}

\begin{proof}
Let $c$ denote the Banach lattice of convergent real sequences and $c_0$
its closed ideal of null sequences. Equip $B=c\oplus c_0$ with the
coordinatewise vector and lattice operations and the norm
\[
 \|(u,v)\|_B=\sup_{n\in\N}\bigl(|u_n|+|v_n|\bigr).
\]
Since $\max(\|u\|_\infty,\|v\|_\infty)\leqslant\|(u,v)\|_B\leqslant
\|u\|_\infty+\|v\|_\infty$, the space $B$ is complete and separable, and
monotonicity of $|u_n|+|v_n|$ in $|u_n|$ and $|v_n|$ makes $B$ a
separable Banach lattice.

For each $m\in\N$, define
\[
 I_m=\bigl\{(u,0)\in B:u_n=0\text{ for all }n>m\bigr\},\qquad
 I=c_0\oplus\{0\}.
\]
Because the order on $B$ is coordinatewise, each $I_m$ is a closed ideal
in $B$, the sequence $(I_m)_{m\in\N}$ is increasing, and
$I=\overline{\bigcup_{m\in\N}I_m}$. Write $Q\colon B\to B/I$,
$Q_m\colon B\to B/I_m$, and $\pi_m\colon B/I_m\to B/I$ for the canonical
quotient maps.

Define a surjective lattice homomorphism $\Phi\colon B\to\R\oplus_\infty c_0$
by
\[
 \Phi(u,v)=\bigl(\lim_{n\to\infty}u_n,\,v\bigr).
\]
Clearly $\ker\Phi=I$. We claim that the induced lattice isomorphism
$\overline\Phi\colon B/I\to\R\oplus_\infty c_0$ is isometric. Indeed, let
$(u,v)\in B$ and put $\lambda=\lim_{n\to\infty}u_n$. For every
$(j,0)\in I$, we have $j\in c_0$, so
\[
 \|(u-j,v)\|_B
 =\sup_{n\in\N}\bigl(|u_n-j_n|+|v_n|\bigr)
 \geqslant\max\bigl(|\lambda|,\|v\|_\infty\bigr),
\]
because $|u_n-j_n|+|v_n|\geqslant|v_n|$ for each $n$ and
$|u_n-j_n|+|v_n|\to|\lambda|$ as $n\to\infty$. Conversely, given
$\eps>0$, choose $N\in\N$ such that $|v_n|\leqslant\eps$ for all $n>N$,
and define $j=(j_n)_{n\in\N}\in c_0$ by $j_n=u_n$ for $1\leqslant n\leqslant N$
and $j_n=u_n-\lambda$ for $n>N$. Then $(j,0)\in I$ and
\[
 \|Q(u,v)\|_{B/I}\leqslant\|(u-j,v)\|_B
 \leqslant\max\bigl(\|v\|_\infty,|\lambda|+\eps\bigr)
 \leqslant\|\Phi(u,v)\|_{\R\oplus_\infty c_0}+\eps.
\]
Since $\eps>0$ was arbitrary, $\overline\Phi$ is isometric. On the other
hand, for every $m\in\N$, every $(u,v)\in B$, and every $n>m$, each
element $(j,0)\in I_m$ has $j_n=0$, whence
\begin{equation}\label{eq:stage-lower-bound}
 \|Q_m(u,v)\|_{B/I_m}
 =\inf_{(j,0)\in I_m}\sup_{k\in\N}\bigl(|u_k-j_k|+|v_k|\bigr)
 \geqslant |u_n|+|v_n|.
\end{equation}

Now choose a norm-one lattice homomorphism $\chi\colon E\to\R$ and define
$T\colon E\oplus_\infty F\to B/I$ by
\[
 T(x,y)=\overline\Phi^{-1}\bigl(\chi(x),\Psi(y)\bigr).
\]
Then $T$ is a contractive lattice homomorphism. Fix $m\in\N$ and suppose
that $S\colon E\oplus_\infty F\to B/I_m$ is a lattice homomorphism (or
merely a positive linear operator) satisfying $\pi_m S=T$. Given
$\delta\in(0,1)$, choose $x_0\in B_E^+$ with $\chi(x_0)\geqslant 1-\delta$.
Choose positive representatives $(u^{(0)},v^{(0)})\in B_+$ and
$(u^{(n)},v^{(n)})\in B_+$ such that
\[
 S(x_0,0)=Q_m\bigl(u^{(0)},v^{(0)}\bigr),\qquad
 S(0,f_n)=Q_m\bigl(u^{(n)},v^{(n)}\bigr)\qquad(n\in\N).
\]
The identities $\overline\Phi(\pi_m S(x_0,0))=(\chi(x_0),0)$ and
$\overline\Phi(\pi_m S(0,f_n))=(0,\Psi f_n)$ yield
\[
 \lim_{n\to\infty}u_n^{(0)}=\chi(x_0)\geqslant 1-\delta
 \qquad\text{and}\qquad
 v_n^{(n)}=(\Psi f_n)_n=1\quad(n\in\N).
\]
For every $n>m$, the element $(x_0,f_n)=(x_0,0)+(0,f_n)$ belongs to the
closed unit ball of $E\oplus_\infty F$. Applying~\eqref{eq:stage-lower-bound}
to $S(x_0,f_n)=Q_m(u^{(0)}+u^{(n)},v^{(0)}+v^{(n)})$ at coordinate $n>m$
gives
\[
 \|S\|\geqslant\|S(x_0,f_n)\|_{B/I_m}
 \geqslant u_n^{(0)}+u_n^{(n)}+v_n^{(0)}+v_n^{(n)}
 \geqslant u_n^{(0)}+1.
\]
Letting $n\to\infty$ and then $\delta\downarrow 0$, we obtain $\|S\|\geqslant 2$.
Thus $T$ admits no lattice homomorphism lift to any stage $B/I_m$ of norm
less than $2$, so $E\oplus_\infty F$ is not semiprojective.
\end{proof}

\begin{corollary}\label{top:linfty-order-continuous}
Let $E$ and $F$ be nonzero separable Banach lattices, and suppose that $F$
has order-continuous norm. If $E\oplus_\infty F$ is semiprojective, then
$F$ is finite-dimensional.
\end{corollary}

\begin{proof}
Suppose that $E\oplus_\infty F$ is semiprojective. As $E$ and $F$ are
contractive lattice retracts of $E\oplus_\infty F$, both are
semiprojective by Proposition~\ref{pre:retract}. By
Corollary~\ref{lift:characters}, $\Hom(E,\R)\ne\{0\}$. If $F$ were
infinite-dimensional, the proof of Theorem~\ref{obs:order-continuous}
would provide a sequence of distinct normalised atoms $(a_n)_{n\in\N}$ in
$F$ and a contractive lattice homomorphism $\Psi\colon F\to c_0$ with
$\Psi(a_n)=e_n$, contradicting Theorem~\ref{top:linfty-obstruction}.
\end{proof}

\begin{corollary}\label{top:extensions-pullbacks}
The Banach lattice $A=\R\oplus_\infty\ell_1$ is $2^+$-projective, but it
is not semiprojective. Consequently:
\begin{enumerate}
\item Semiprojectivity and $1^+$-projectivity are not preserved by finite
$\ell_\infty$-sums.
\item Semiprojectivity and $1^+$-projectivity are not preserved by
extensions: in the short exact sequence
\[
 0\longrightarrow\ell_1\longrightarrow\R\oplus_\infty\ell_1
 \longrightarrow\R\longrightarrow 0,
\]
both the closed ideal $\ell_1$ and the quotient $\R$ are separable and
$1^+$-projective, whereas the middle lattice is not semiprojective.

\end{enumerate}
\end{corollary}

\begin{proof}
The formal identity map between $\R\oplus_\infty\ell_1$ and
$\R\oplus_1\ell_1\cong\ell_1$ is a lattice isomorphism of distortion $2$.
Since $\ell_1$ is $1^+$-projective, $\R\oplus_\infty\ell_1$ is
$2^+$-projective. Taking $E=\R$ and $F=\ell_1$ with $f_n=e_n$ and the
formal inclusion $\Psi\colon\ell_1\to c_0$ in
Theorem~\ref{top:linfty-obstruction} shows that $\R\oplus_\infty\ell_1$ is
not semiprojective. Assertions~(1) and~(2) follow immediately.

\end{proof}
\section{Openness of the set of liftable homomorphisms}\label{open:section}

Following the terminology used for $C^*$-algebras
\cite{oyetunbi-tikuisis}, we consider the topology of liftable maps.
The norm convention matters: the threshold below is $1+\eps$ for
every contractive map, including maps whose norm is strictly less than one.

\begin{definition}\label{open:definition}
Let $Q\colon B\to B/I$ be a quotient by a closed ideal. A contractive
lattice homomorphism $T\colon A\to B/I$ is \emph{$1^+$-liftable} if,
for every $\eps>0$, there is a lattice homomorphism $S\colon A\to B$
with $QS=T$ and $\|S\|\leqslant1+\eps$. Write
\[
\Hom_1(A,B/I)=\{T\in\Hom(A,B/I):\|T\|\leqslant1\}
\]
and equip it with the point-norm topology. We call $A$ \emph{$\ell$-open}
(respectively, \emph{$\ell$-closed}) if, for every separable $B$ and
closed ideal $I\subseteq B$, the $1^+$-liftable maps form an open
(respectively, closed) subset of $\Hom_1(A,B/I)$.
\end{definition}

\begin{theorem}\label{open:equivalence}
A Banach lattice is $\ell$-open if and only if it is $1^+$-projective
relative to separable ambient lattices. Consequently, every
$\ell$-open Banach lattice is semiprojective and $\ell$-closed.
\end{theorem}

\begin{proof}
Suppose that $A$ is $\ell$-open. Fix a separable $B$, a closed ideal
$I\subseteq B$, and a contraction $T\colon A\to B/I$. Let
\[
\mu(T)=\inf\{\|S\|:S\colon A\to B\text{ is a lattice homomorphism},\ QS=T\},
\]
where the infimum of the empty set is $+\infty$.
The path $[0,1]\ni t\mapsto tT\in\Hom_1(A,B/I)$ is continuous.
The zero map is liftable, so openness supplies some $t>0$ for which
$tT$ is $1^+$-liftable. Rescaling a lift shows that $\mu(T)<\infty$.
For every $t>0$, positive homogeneity gives
\[
\mu(tT)=t\mu(T),
\]
and $tT$ is $1^+$-liftable precisely when $t\mu(T)\leqslant1$.
The latter equivalence uses the infimum of the lift norms; it does
not require a lift attaining that infimum.

If $\mu(T)>1$, the inverse image of the liftable set under this
path is $[0,1/\mu(T)]$, which is not open in $[0,1]$. This contradicts
$\ell$-openness. Thus $\mu(T)\leqslant1$, and every contraction is
$1^+$-liftable. Normalising an arbitrary nonzero bounded lattice
homomorphism proves the asserted relative projectivity.

Conversely, relative $1^+$-projectivity makes every element of
$\Hom_1(A,B/I)$ liftable; the liftable set is therefore both open
and closed. It also implies semiprojectivity: lift a given map
$A\to B/I$ into $B$ and then compose with $B\to B/I_1$.
The stage can always be chosen to be $1$.
\end{proof}

\begin{corollary}\label{open:separable}
A separable Banach lattice is $\ell$-open if and only if it is
$1^+$-projective.
\end{corollary}

\begin{proof}
Combine Theorem~\ref{open:equivalence} with
Proposition~\ref{lift:reduction}.
\end{proof}

\begin{remark}\label{open:remaining-question}
Thus the implication from $\ell$-openness to semiprojectivity holds
without a separability assumption on the domain. For separable
domains, the converse is exactly the remaining question whether
semiprojectivity implies $1^+$-projectivity.
\end{remark}

Closedness also has concrete consequences when the identity can be
approximated by lattice homomorphisms of finite rank. The approximation
required here is stronger than the usual approximation property, since
the approximating operators must preserve modulus.

\Needspace{9\baselineskip}
\begin{proposition}\label{open:finite-rank}
Suppose that a Banach lattice $A$ admits contractive finite-rank lattice
homomorphisms $R_n\colon A\to A$ such that $R_na\to a$ for every
$a\in A$. Then the following conditions are equivalent:
\begin{enumerate}
\item $A$ is $\ell$-closed;
\item $A$ is $\ell$-open;
\item $A$ is $1^+$-projective.
\end{enumerate}
\end{proposition}

\begin{proof}
The approximation assumption implies that $A$ is separable: the union
of the finite-dimensional ranges of the $R_n$ has dense linear span.
By Theorem~\ref{open:equivalence} and Corollary~\ref{open:separable},
it remains to prove that $\ell$-closedness implies $\ell$-openness.

Suppose that $A$ is $\ell$-closed. Fix a separable $B$, a closed ideal
$I\subseteq B$, and a contraction $T\colon A\to B/I$. The range
$E_n=R_n(A)$, with its inherited norm and order, is a finite-dimensional
Banach lattice. Such lattices are $1^+$-projective \cite{dpw}.
Consequently, $T|_{E_n}$ has, for every $\eps>0$, a lattice lift
$S_{n,\eps}\colon E_n\to B$ of norm at most $1+\eps$. The composition
$S_{n,\eps}R_n$ lifts $TR_n$ with the same bound. Thus every $TR_n$
is $1^+$-liftable, while $TR_n\to T$ in the point-norm topology.
Closedness implies that $T$ is $1^+$-liftable. Since $T$ was arbitrary,
Theorem~\ref{open:equivalence} proves $\ell$-openness.
\end{proof}

The coordinate truncations of $c_0$ and $\ell_p$, $1<p<\infty$, satisfy
the approximation hypothesis. Their failure of semiprojectivity therefore
also implies that they are neither $\ell$-open nor $\ell$-closed.

\section{Further questions}\label{sec:questions}
The results above leave a specific gap between almost isometric lifting and bounded lifting.
\begin{question}\label{q:separable}
Is every separable semiprojective Banach lattice $1^+$-projective?
\end{question}
By Theorem~\ref{main:D}, this is equivalent to asking whether every separable semiprojective lattice is $\ell$-open. Theorem~\ref{main:C} supplies a finite lifting constant, but it does not make that constant arbitrarily close to one. Note that $\R\oplus_\infty\ell_1$ is $2^+$-projective and not $1^+$-projective, but by Corollary~\ref{top:extensions-pullbacks} it is not semiprojective either, so it does not provide a counterexample to Question~\ref{q:separable}. Likewise, $\ell_1(\Gamma)$ and $\R\oplus_1\FBL(\Gamma)$ for uncountable $\Gamma$ do not settle this question, because their separation from unrestricted projectivity depends on the restriction to separable ambient lattices.

Theorem~\ref{main:B} settles finite $\ell_1$-sum permanence, whereas Corollary~\ref{top:extensions-pullbacks} shows that finite $\ell_\infty$-sums and extensions of $1^+$-projective Banach lattices need not be semiprojective. For infinite $\ell_1$-sums, the following case remains open.
\begin{question}\label{q:sums}
Are countable $\ell_1$-sums of semiprojective Banach lattices semiprojective?
\end{question}
For countably many semiprojective summands the disjoint-ideal construction in Section~\ref{sums:section} still applies, but the individual lifting stages need not have a common finite upper bound. If the summands are $\ell$-open, Theorem~\ref{main:D} and Theorem~\ref{main:B}(2) give a positive answer, even for arbitrary $\ell_1$-sums; for separable summands, Question~\ref{q:sums} is therefore closely linked to Question~\ref{q:separable}.

\begin{question}\label{q:limits}
Is every separable Banach lattice a sequential inductive limit of separable semiprojective Banach lattices, with contractive lattice homomorphisms as connecting maps?
\end{question}
Here the inductive limit is taken in the category of Banach lattices and contractive lattice homomorphisms, including completion in its norm. The question asks for a structural approximation by semiprojective lattices, rather than for the quotient-tower lifting property in Definition~\ref{pre:semi}. Related questions and closure results for $C^*$-algebras are studied in \cite{blackadar85,thiel-inductive}.

\subsection*{Acknowledgements}
We are grateful to Leonid Positselski (Prague) for drawing our attention to the connection with finitely presentable objects in the sense of Ad\'amek and Rosick\'y. The second-named author gratefully acknowledges support received from NCN Sonata-Bis~13 (2023/50/E/ST1/00067).

\end{document}